\documentclass[12pt]{amsart}
\usepackage{amssymb, amsthm, enumerate, graphicx, amscd, amsmath, amssymb, amsfonts, color}

\numberwithin{equation}{section}
\newtheorem{theorem}[equation]{Theorem}

\newtheorem{proposition}[equation]{Proposition}

\title{Orthogonal projections in the local Dirichlet spaces}

\author[Fricain]{Emmanuel Fricain}
 \address{Univ. Lille, CNRS, UMR 8524 - Laboratoire Paul Painlev\'e, F-59000 Lille, France.}
 \email{emmanuel.fricain@univ-lille.fr}

\author[Mashreghi]{Javad Mashreghi}
 \address{D\'epartement de math\'ematiques et de statistique,
         Universit\'e Laval,
         Qu\'ebec, QC,
         Canada G1K 7P4.}
 \email{Javad.Mashreghi@mat.ulaval.ca}

\thanks{This work was partially supported by the NSERC-Discovery research grant, the Canada Research Chairs program, and by the Labex CEMPI  (ANR-11-LABX-0007-01).}

\keywords{Local Dirichlet space, orthonormal basis, difference quotient operator}

\subjclass[2020]{33C45, 33C47, 42B35}

\begin{document}

\begin{abstract}
We present an explicit formula for the orthogonal projection onto the subspace of analytic polynomials of degree at most $n$ in the local Dirichlet space $\mathcal{D}_\mu$, where the positive measure $\mu$ consists of a finite number of Dirac measures located at points on the unit circle $\mathbb{T}$. This result has two key aspects: first, while it is known that polynomials are dense in $\mathcal{D}_\mu$, this approach offers a concrete linear approximation scheme within the space. Second, due to the orthogonality of the polynomials involved, the scheme is qualitative, as the distance of an arbitrary function  $f \in \mathcal{D}_\mu$ to the projected subspace is explicitly determined.
\end{abstract}

\maketitle

\section{Introduction}
Let $\mathbb{D}$ be the open unit disk in the complex plane, and let $\mathbb{T}$ denote its boundary. For each $f \in \mbox{Hol}(\mathbb{D})$, the family of analytic functions on $\mathbb{D}$, the local Dirichlet integral of $f$ is given by
\begin{equation}\label{E:def-D-zeta}
\mathcal{D}_{\zeta}(f) = \frac{1}{\pi} \int_{\mathbb{D}} |f'(z)|^2 \, \frac{1-|z|^2}{|\zeta-z|^2} \, dA(z),
\end{equation}
where $\zeta \in \mathbb{T}$ and $dA(z) = dx \, dy$ is the planar Lebesgue measure. The local Dirichlet space $\mathcal{D}_{\zeta}$ consists of all functions $f \in \mbox{Hol}(\mathbb{D})$ for which $\mathcal{D}_{\zeta}(f) < \infty$. It is known that $f \in \mbox{Hol}(\mathbb{D})$ belongs to $\mathcal{D}_{\zeta}$ if and only if
\begin{equation}\label{E:rep-dzeta}
f(z) = a + (z-\zeta)g(z),
\end{equation}
where $g \in H^2$ and $a \in \mathbb{C}$. Moreover, we have $\mathcal{D}_{\zeta}(f) = \|g\|_{H^2}^2$, as seen in \cite{MR1116495} and \cite[page 111]{MR3185375}. This representation ensures that
\begin{equation}\label{E:bdry-f}
f(\zeta) := \lim_{r \to 1^-} f(r\zeta) = a.
\end{equation}
Thus, the difference quotient operator $Q_{\zeta}: \mathcal{D}_{\zeta} \to H^2$,
\[
(Q_{\zeta}f)(z) := \frac{f(z)-f(\zeta)}{z-\zeta}, \qquad z \in \mathbb{D},
\]
is well-defined. We can then equip $\mathcal{D}_{\zeta}$ with the norm
\begin{equation}\label{E:norm-dzeta-2}
\|f\|_{\mathcal{D}_{\zeta}}^{2} := |f(\zeta)|^{2} + \|Q_{\zeta}f\|_{H^2}^{2}.
\end{equation}
Using the polarization identity, we also have the inner product
\begin{equation}\label{inner-dzeta-2}
\langle f_1,f_2 \rangle_{\mathcal{D}_{\zeta}} = f_1(\zeta) \, \overline{f_2(\zeta)} + \langle Q_{\zeta}f_1,Q_{\zeta}f_2 \rangle_{H^2}.
\end{equation}

The local Dirichlet spaces were systematically studied by Richter and Sundberg \cite{MR1116495, MR1259923, MR1145733} in relation to invariant subspaces of the shift operator on the classical Dirichlet space. These spaces have also appeared in various other studies \cite{MR1396993, MR1945292, MR3540327}. Although our focus in this paper is primarily on local Dirichlet spaces, we must also consider a generalization given by
\begin{equation}\label{E:def-D-zeta-gen}
\mathcal{D}_{\mu}(f) = \frac{1}{\pi} \int_{\mathbb{D}} |f'(z)|^2 \, P\mu(z) \, dA(z),
\end{equation}
where
\[
P\mu(z) = \int_{\mathbb{T}} \frac{1-|z|^2}{|z-\zeta|^2} \, d\mu(\zeta), \qquad z \in \mathbb{D},
\]
is the Poisson integral of $\mu$. A detailed description of these spaces can be found in \cite[Ch. 7]{MR3185375}. Note that $P\mu$ is a positive harmonic function on $\mathbb{D}$, and by Herglotz’s theorem \cite{MR2500010}, any positive harmonic function has such a representation. S. Richter demonstrated that polynomials are dense in $\mathcal{D}_\mu$ \cite{MR1013337}. The estimates required for the Dirichlet integral of dilations $f_r$ were further refined by Richter and Sundberg \cite{MR1116495}, Aleman \cite{MR1079693}, and Sarason \cite{MR1396993}.

The rest of the paper is organized as follows. In Section \ref{S:main}, we state the main results which consist of Theorems \ref{T:poly-taylor-2} and \ref{T:poly-taylor-22}. Section \ref{S:poly-ob-3} contains a brief discussion of $\mathcal{D}_\mu$ and introduces the orthogonal polynomials that we exploit. A brief introduction of the complete homogeneous symmetric polynomial of degree $k$ in $s$ variables $\zeta_0,\dots,\zeta_{s-1}$ and the homogeneous symmetric polynomials is presented in Section \ref{S:sym-poly}. Finally, in Sections \ref{S:proof1}, we provide the proof of the main results.

\section{The main result} \label{S:main}
For an analytic function
\[
f(z) = \sum_{j=0}^{\infty} \hat{f}(j) z^j \in \mbox{Hol}(\mathbb{D}),
\]
define
\begin{eqnarray}
\Delta(n,j) &=& \sum_{k=2j-n+s}^{j+s} \left( \sum_{m=j-n+s}^{k-j} (-1)^{m} T_{m} S_{k-j-m} \right) \hat{f}(k) \notag \\
&+& \sum_{k=j+s+1}^{\infty} \left( \sum_{m=j-n+s}^{s} (-1)^{m} T_{m} S_{k-j-m} \right) \hat{f}(k), \label{E:formual-Delta}
\end{eqnarray}
where the indices $n$ and $j$ are in the ranges $n \geq s$ and $n-s+1 \leq j \leq n$. The symmetric polynomials $S_k = S_k(z_1,\dots,z_{s})$ and $T_k = T_k(z_1,\dots,z_{s})$, used in the expressions above, are discussed in further details in Section \ref{S:sym-poly}. Moreover, the parameters $z_1, \dots, z_s$ are points in the spectrum of the measure $\mu$, denoted by $\zeta_0, \dots, \zeta_{s-1}$ for convenience.

The following result provides an explicit formula for the orthogonal projection $\mathbf{P}_n$ onto the subspace $\mathcal P_n$ of polynomials of degree at most $n$ in the harmonically weighted Dirichlet space $\mathcal{D}_\mu$. Here, $c_j > 0$ and $\mu = c_0 \delta_{\zeta_0} + \cdots + c_{s-1} \delta_{\zeta_{s-1}}$ is the discrete measure anchored at the points $\zeta_0, \zeta_1, \dots, \zeta_{s-1}$ on $\mathbb{T}$.

\begin{theorem} \label{T:poly-taylor-2}
Let $f(z) = \sum_{j=0}^{\infty} \hat{f}(j) z^j \in \mathcal{D}_\mu$, and let, for $n \geq s$ and $n-s+1 \leq j \leq n$, define $\Delta(n,j)$ as in \eqref{E:formual-Delta}. Then, for $n \geq s$,
\[
\mathbf{P}_nf(z)  = \sum_{j=0}^{n-s} \hat{f}(j) z^j + \sum_{j=n-s+1}^{n} \Delta(n,j) z^j.
\]
and $\mathbf{P}_nf \to f$ in $\mathcal{D}_\mu$-norm.
\end{theorem}

The result is further complemented by an explicit formula for the distance between any element of $\mathcal{D}_\mu$ and the subspace $\mathcal P_n$ of polynomials of degree at most $n$, based on the Taylor coefficients of $f$. Recall that the distance of a point $x$ to a set $M$ in a normed space is defined as
\[
\mbox{dist}(x,M) := \inf_{y \in M} \|x - y\|.
\]
This formula provides a clear way to compute the distance in terms of the given function's Taylor series.

\begin{theorem} \label{T:poly-taylor-22}
With the notation of Theorem \ref{T:poly-taylor-2}, for $n \geq s$, we have
\[
\mbox{dist}_{\mathcal{D}_\mu}(f,\mathcal{P}_n)
= \left\{ \sum_{j=n+1}^{\infty} \left| \sum_{k=j}^{\infty} \hat{f}(k) S_{k-j}(\zeta_0,\dots,\zeta_{s-1}) \right|^2 \right\}^{1/2}.
\]
\end{theorem}

The special case of $\mathcal{D}_\zeta=\mathcal{D}_{\delta_\zeta}$, which consists of a single Dirac measure anchored at one point on $\mathbb{T}$, was the initial focus of such studies. A direct proof, bypassing the concept of orthogonality, first demonstrated in \cite{MR4240772} that $\mathbf{P}_n f \to f$ in the $\mathcal{D}_\zeta$-norm. Later, this result was rederived with emphasis on orthogonal polynomials and further extended to cases involving measures composed of a finite number of Dirac measures on $\mathbb{T}$ in \cite{EF-JM-CAOT}. However, in the earlier approach, the norm on $\mathcal{D}_\mu$ was defined as
\[
\|f\|_{\mathcal{D}_{\zeta}}^{2} = \|f\|^{2} + \|Q_{\zeta}f\|_{H^2}^{2},
\]
which is equivalent to \eqref{E:norm-dzeta-2}.

\section{Discrete measures with finite support} \label{S:poly-ob-3}
Let $\zeta_0,\zeta_1,\dots,\zeta_{s-1}$ be $s$ distinct points on $\mathbb T$ (to overcome a notational difficulty in upcoming formulas, we start indexing from zero here). Let $c_j > 0$ and consider the discrete measure
\[
\mu = c_0 \delta_{\zeta_0} + \cdots + c_{s-1} \delta_{\zeta_{s-1}}.
\]
As a straightforward generalization of \eqref{E:rep-dzeta}, each $f \in \mathcal{D}_\mu$ has the unique representation
\begin{equation}\label{E:rep-dzeta-2}
f(z) = a_0+a_1z+\cdots+a_{s-1}z^{s-1} + (z-\zeta_0)(z-\zeta_1)\cdots(z-\zeta_{s-1})g(z),
\end{equation}
where $g \in H^2$. It is easy to see that
\[
g := Q_{\zeta_0}Q_{\zeta_1}\cdots Q_{\zeta_{s-1}}f, \qquad f \in \mathcal{D}_{\mu}.
\]
Therefore, we define the operator
\[
Q := Q_{\zeta_0}Q_{\zeta_1}\cdots Q_{\zeta_{s-1}}
\]
for which the ordering of $Q_{\zeta_j}$ is irrelevant. For simplicity of notations, put
\begin{equation}\label{E:rep-dzeta-234}
q(z):=(z-\zeta_0)(z-\zeta_1)\cdots(z-\zeta_{s-1})
\end{equation}
and hence rewrite \eqref{E:rep-dzeta-2} as
\begin{equation}\label{E:rep-dzeta-23}
f(z) = a_0+a_1z+\cdots+a_{s-1}z^{s-1} + q(z) (Qf)(z).
\end{equation}

In the light of \eqref{E:norm-dzeta-2}, a possible norm on $\mathcal{D}_\mu$ is
\begin{equation}\label{E:norm-dzeta-3}
\|f\|_{\mathcal{D}_{\mu}}^{2} := \sum_{k=0}^{s-1}|f(\zeta_k)|^{2} + \|Qf\|_{H^2}^2,
\end{equation}
and the corresponding inner product is
\begin{equation}\label{E:inner-dzeta-3}
\langle f,g \rangle_{\mathcal{D}_{\mu}} := \sum_{k=0}^{s-1}f(\zeta_k) \overline{g(\zeta_k)} + \langle Qf,Qg \rangle_{H^2}.
\end{equation}
However, there are other possibilities to add a term to the semi-norm $\|Qf\|_{H^2}$ and create a genuine norm. Three such possibilities are studied in \cite{MR4507242}. All such norms are equivalent, but with respect to \eqref{E:norm-dzeta-3} the calculations are easier.

With respect to the above inner product, the orthogonal basis of polynomials is straightforward to find. First, let
\[
p_{k}(z) := \frac{q(z)}{(z-\zeta_{k})q'(\zeta_{k})}, \qquad 0 \leq k \leq s-1.
\]
Then $Qp_{k}=0$, $p_{k}(\zeta_j) \overline{p_{\ell}(\zeta_j)} =0$ and $p_{k}(\zeta_k) =1$, for all $k,\ell,j \in \{0,1,\dots,s-1\}$ with $k \ne \ell$. Therefore, $p_0,p_1,\dots,p_{s-1}$ is an orthonormal set of polynomials whose span is $\mathcal{P}_{s-1}$. However, note that $\mbox{deg}(p_{0})=\cdots = \mbox{deg}(p_{s-1})= s-1$. Then put
\begin{equation}\label{E:formul-pj-s}
p_j(z) = q(z) z^{j-s}, \qquad j \geq s.
\end{equation}
Then by direct verification, we see that $(p_j)_{j \geq 0}$ is an orthonormal basis for $\mathcal{D}_\mu$.

\section{Symmetric polynomials} \label{S:sym-poly}
To derive an explicit formula for the orthogonal projection onto the subspace of polynomials of degree at most $n$, we need to recall some key facts from the theory of symmetric polynomials. Two primary types of symmetric polynomials of degree $k$ in $s$ variables $z_1, \dots, z_{s}$ are as follows:

\begin{enumerate}[(i)]
\item The {\em complete homogeneous symmetric polynomial} is given by
\[
S_k(z_1,\dots,z_{s}) := \sum_{1 \leq i_1 \leq i_2 \leq \cdots \leq i_k \leq s} z_{i_1}z_{i_2}\cdots z_{i_k}, \qquad k \geq 1,
\]
with $S_0(z_1,\dots,z_{s}) := 1$. The sum consists of all monomials of degree $k$ that can be formed with the variables $z_1, \dots, z_s$, allowing repetition.

\item The {\em homogeneous symmetric polynomial} is defined as
\[
T_k(z_1,\dots,z_{s}) := \sum_{1 \leq i_1 < i_2 < \cdots < i_k \leq s} z_{i_1}z_{i_2}\cdots z_{i_k}, \qquad k \geq 1,
\]
with $T_0(z_1,\dots,z_{s}) := 1$. In this case, repetition is not allowed, so $k$ can be at most $s$.
\end{enumerate}

Note that $S_0 = T_0$ and $S_1 = T_1$, but for $k \geq 2$, $S_k \neq T_k$.

The first encounter with symmetric polynomials arises during the expansion of $q$ in \eqref{E:rep-dzeta-234}. We can express $q$, which depends on the parameters $\zeta_0, \dots, \zeta_{s-1}$ and the variable $z$, as follows:
\begin{equation}\label{E:formul-q}
q(z) = \sum_{j=0}^{s} (-1)^{s-j} T_{s-j}(\zeta_0, \dots, \zeta_{s-1}) z^j.
\end{equation}
For simplicity, we will use $S_k$ and $T_k$ throughout with the understanding that the arguments are $\zeta_0, \dots, \zeta_{s-1}$.

There are numerous identities that highlight various features of symmetric polynomials. For a detailed treatment of symmetric polynomials, see \cite{MR3443860}. In our context, we rely on two well-known identities:
\begin{eqnarray}
\sum_{m=0}^{k} (-1)^{m} T_{m} S_{k-m} &=& 0, \qquad 1 \leq k \leq s, \label{E:pol-sym-1}\\
\sum_{m=0}^{s} (-1)^{m} T_{m} S_{k-m} &=& 0, \qquad k \geq s+1. \label{E:pol-sym-2}
\end{eqnarray}

\section{Proofs of Theorem \ref{T:poly-taylor-2} and Theorem \ref{T:poly-taylor-22}} \label{S:proof1}

\begin{proof}[Proof of Theorem \ref{T:poly-taylor-2}]
Let $f\in\mathcal D_\mu$. There are $g \in H^2$ and coefficients $a_0,\dots,a_{s-1}$ such that \eqref{E:rep-dzeta-2} holds. Since $p_0,\dots,p_{s-1}$ also spans $\mathcal{P}_{s-1}$, there are coefficients $b_0,\dots,b_{s-1}$ such that
\[
a_0+a_1z+\cdots+a_{s-1}z^{s-1} = b_0p_0(z)+\cdots+b_{s-1}p_{s-1}(z).
\]
Since $p_j(\zeta_k)=\delta_{jk}$, $0\leq j,k\leq s-1$, the coefficients $b_k$ are uniquely determined via the system
\[
\begin{pmatrix}
b_0\\
b_1\\
\vdots\\
b_{s-1}
\end{pmatrix}
=
\begin{pmatrix}
1 & \zeta_0 & \cdots & \zeta_0^{s-1}\\
1 & \zeta_1 & \cdots & \zeta_1^{s-1}\\
\vdots & \vdots & \ddots & \vdots \\
1 & \zeta_{s-1} & \cdots & \zeta_{s-1}^{s-1}\\
\end{pmatrix}
\begin{pmatrix}
a_0\\
a_1\\
\vdots\\
a_{s-1}
\end{pmatrix}.
\]
Now, using the Taylor series expansion of $g$, we see that
\begin{eqnarray}
f(z) &=& a_0+a_1z+\cdots+a_{s-1}z^{s-1} + q(z) \sum_{j=0}^{\infty} \hat{g}(j) z^j \notag\\
&=& b_0p_0(z)+\cdots+b_{s-1}p_{s-1}(z) + \sum_{j=s}^{\infty} \hat{g}(j-s) p_{j}(z), \label{E:an-gn-general}
\end{eqnarray}
which is the decomposition of $f$ with respect to the orthonormal basis $(p_n)_{n\geq0}$ and, moreover, the series converges in $\mathcal{D}_\mu$-norm. To express $b_j$ and $\hat{g}(j)$ in terms of $\hat{f}(j)$, we exploit further \eqref{E:an-gn-general}. First, plug $z=\zeta_k$, and recall that $p_j(\zeta_k)=\delta_{jk}$, $0\leq j,k\leq s-1$ and $p_j(\zeta_k)=0$, $j\geq s$, to obtain
\[
b_k = f(\zeta_k) = \sum_{j=0}^{\infty} \hat{f}(j) \zeta_k^j, \qquad 0 \leq k \leq s-1.
\]
Second, apply $Q$ to both sides of \eqref{E:an-gn-general}. Since $Qp_j=0$, $0 \leq j \leq s-1$ and $Qp_j=z^{j-s}$, $j \geq s$, we get
\begin{equation}\label{E:g-n-fn-Q1}
\sum_{j=0}^{\infty} \hat{g}(j) z^j = Qf(z).
\end{equation}
To explicitly write $\hat{g}(n)$ in terms of $\hat{f}(n)$, we consider the case where just one Dirac measure at $\zeta$ is present. In that case the formula becomes
\[
\begin{pmatrix}
\hat{g}(0)\\
\hat{g}(1)\\
\hat{g}(2)\\
\hat{g}(3)\\
\hat{g}(4)\\
\vdots\\
\end{pmatrix}
=
\begin{pmatrix}
0 & 1 & \zeta & \zeta^2 & \zeta^3 & \cdots \\
0 & 0 & 1 & \zeta & \zeta^2 & \cdots \\
0 & 0 & 0 & 1 & \zeta & \cdots \\
0 & 0 & 0 & 0 & 1 & \cdots \\
0 & 0 & 0 & 0 & 0 & \cdots \\
\vdots & \vdots & \vdots & \vdots & \vdots & \ddots \\
\end{pmatrix}
\begin{pmatrix}
\hat{f}(0)\\
\hat{f}(1)\\
\hat{f}(2)\\
\hat{f}(3)\\
\hat{f}(4)\\
\vdots\\
\end{pmatrix}.
\]
Call the right hand side matrix $M_{\zeta}$. Since $Q = Q_{\zeta_0}Q_{\zeta_1}\cdots Q_{\zeta_{s-1}}$ we need to find
\[
M := M_{\zeta_0}M_{\zeta_1}\cdots M_{\zeta_{s-1}}.
\]
A simple calculation reveals that
\[
M = \begin{pmatrix}
0 & \cdots & 0 & 1 & S_1 & S_2 & S_3 & \cdots \\
0 & \cdots & 0 & 0 & 1 & S_1 & S_2 & \cdots \\
0 & \cdots & 0 & 0 & 0 & 1 & S_1 & \cdots \\
0 & \cdots & 0 & 0 & 0 & 0 & 1 & \cdots \\
\vdots & \ddots & \vdots & \vdots & \vdots & \vdots & \vdots & \ddots \\
\end{pmatrix},
\]
where the first $s$ columns are zero. Therefore, by removing the zero columns, \eqref{E:g-n-fn-Q1} implies
\[
\begin{pmatrix}
\hat{g}(0)\\
\hat{g}(1)\\
\hat{g}(2)\\
\hat{g}(3)\\
\vdots\\
\end{pmatrix}
=
\begin{pmatrix}
S_0 & S_1 & S_2 & S_3 & \cdots \\
0 & S_0 & S_1 & S_2 & \cdots \\
0 & 0 & S_0 & S_1 & \cdots \\
0 & 0 & 0 & S_0 & \cdots \\
\vdots & \vdots & \vdots & \vdots & \ddots \\
\end{pmatrix}
\begin{pmatrix}
\hat{f}(s)\\
\hat{f}(s+1)\\
\hat{f}(s+2)\\
\hat{f}(s+3)\\
\vdots\\
\end{pmatrix}.
\]
In other words,
\begin{equation}\label{E:g-n-fn-Q2}
\hat{g}(j) = \sum_{k=s+j}^{\infty} \hat{f}(k) S_{k-s-j}(\zeta_0,\dots,\zeta_{s-1}), \qquad j \geq 0.
\end{equation}
Hence, the orthonormal decomposition of $f$ with respect to the orthonormal basis $(p_j)_{j\geq0}$ is
\begin{equation}\label{E:decomp-pn-2}
f = \sum_{j=0}^{s-1} \left( \sum_{k=0}^{\infty} \hat{f}(k) \zeta_j^k \right) p_{j} + \sum_{j=s}^{\infty} \left( \sum_{k=j}^{\infty} \hat{f}(k) S_{k-j}(\zeta_0,\dots,\zeta_{s-1}) \right) p_{j}.
\end{equation}
Therefore, the orthogonal projection of $f \in \mathcal{D}_\mu$ onto $\mathcal{P}_n$, $n \geq s$, is given by
\begin{equation}\label{E:decomp-pn-232}
\mathbf{P}_nf = \sum_{j=0}^{s-1} \left( \sum_{k=0}^{\infty} \hat{f}(k) \zeta_j^k \right) p_{j} + \sum_{j=s}^{n} \left( \sum_{k=j}^{\infty} \hat{f}(k) S_{k-j}(\zeta_0,\dots,\zeta_{s-1}) \right) p_{j}.
\end{equation}

Finally, we proceed to provide another formula for $\mathbf{P}_nf$ based on monomials $z^n$. By \eqref{E:formul-pj-s} and \eqref{E:formul-q},
\begin{eqnarray*}
&& \sum_{j=s}^{n} \left( \sum_{k=j}^{\infty} \hat{f}(k) S_{k-j} \right) p_{j}(z) = \sum_{j=s}^{n} \left( \sum_{k=j}^{\infty} \hat{f}(k) S_{k-j} \right) q(z)z^{j-s}\\
&=& \sum_{j=s}^{n} \left( \sum_{k=j}^{\infty} \hat{f}(k) S_{k-j} \right) \left( \sum_{m=0}^{s} (-1)^{s-m} T_{s-m} z^{m+j-s} \right)\\
&=& \sum_{j=s}^{n} \left( \sum_{k=j}^{\infty} \hat{f}(k) S_{k-j} \right) \left( \sum_{\ell=j-s}^{j} (-1)^{j-\ell} T_{j-\ell} z^{\ell} \right).
\end{eqnarray*}
Now, we face with an interesting phenomenon. For $s \leq j \leq n-s$, the coefficient of $z^j$ is
\[
\sum_{m=0}^{s}\left( \sum_{k=j+m}^{\infty} \hat{f}(k) S_{k-j-m} \right) (-1)^{m} T_{m},
\]
which after changing the summation becomes
\[
\sum_{k=j}^{j+s} \left( \sum_{m=0}^{k-j} (-1)^{m} T_{m} S_{k-j-m} \right)  \hat{f}(k)
+ \sum_{k=j+s+1}^{\infty} \left( \sum_{m=0}^{s} (-1)^{m} T_{m} S_{k-j-m} \right)  \hat{f}(k).
\]
However, with a straight change of index, by \eqref{E:pol-sym-1},
\[
\sum_{m=0}^{k-j} (-1)^{m} T_{m} S_{k-j-m}=0, \qquad j+1 \leq k \leq j+s,
\]
and, by \eqref{E:pol-sym-2},
\[
\sum_{m=0}^{s} (-1)^{m} T_{m} S_{k-j-m}=0, \qquad k \geq j+s+1.
\]
Therefore, the only surviving term corresponds to $k=j$ for which we get
\[
\sum_{m=0}^{k-j} (-1)^{m} T_{m} S_{k-j-m}= T_0 S_0 =1.
\]
Therefore, in the polynomial $\mathbf{P}_nf$, the coefficient of $z^j$ is $\hat{f}(j)$ at least for $s \leq j \leq n-s$. For $j <s$, since on one hand the coefficients of $\mathbf{P}_nf$ do not depend on $n$, and on the other hand they have to converge to the corresponding coefficients of $f$, we conclude that they are also equal to $\hat{f}(j)$. However, with similar reasoning, for $n-s+1 \leq j \leq n$, the coefficient of $z^j$ is
\[
\sum_{m=j-n+s}^{s}\left( \sum_{k=j+m}^{\infty} \hat{f}(k) S_{k-j-m} \right) (-1)^{m} T_{m},
\]
which after changing the summation becomes $\Delta(n,j)$. Therefore, the formula for $\mathbf{P}_nf$ follows.

\end{proof}

\begin{proof}[Proof of Theorem \ref{T:poly-taylor-22}]
This proof is easy since the essential work is already done. The identity \eqref{E:decomp-pn-2} provides the orthonormal decomposition of $f$ with respect to the orthonormal basis $(p_j)_{j\geq0}$, and thus $\mathbf{P}_nf$ is given by \eqref{E:decomp-pn-232}.  Therefore, by orthonormality of $p_n$, we immediately deduce that, for $n\geq s$, we have
\begin{eqnarray*}
\mbox{dist}_{\mathcal{D}_\mu}(f,\mathcal{P}_n) &=& \|f-\mathbf{P}_nf\|_{\mathcal D_\mu} \\
&=&\left\| \sum_{j=n+1}^{\infty} \left( \sum_{k=j}^{\infty} \hat{f}(k) S_{k-j}(\zeta_0,\dots,\zeta_{s-1}) \right) p_{j} \right\|_{\mathcal D_\mu}\\
&=& \left\{ \sum_{j=n+1}^{\infty} \left| \sum_{k=j}^{\infty} \hat{f}(k) S_{k-j}(\zeta_0,\dots,\zeta_{s-1}) \right|^2 \right\}^{1/2}.
\end{eqnarray*}
This completes the proof or theorem.

\noindent Remark: Using  \eqref{E:decomp-pn-232}, and the fact that, for every $0\leq \ell\leq s-1$,
\[
p_j(\zeta_\ell)=\begin{cases}
1&\mbox{if }j=\ell\mbox{ and } 0\leq j\leq s-1\\
0&\mbox{otherwise}
\end{cases},
\]
we get that
\[
(\mathbf{P}_nf)(\zeta_\ell)=\sum_{k=0}^\infty \hat f(k)\zeta_\ell^k=f(\zeta_\ell).
\]
In particular, $\mathbf{P}_nf=P_n^\Gamma f$, where $P_n^\Gamma$ is precisely the polynomial presented in \cite{EF-JM-CAOT} for $\Gamma=\{\zeta_0,\dots,\zeta_{s-1}\}$.

\end{proof}

\bibliographystyle{plain}
\bibliography{CRreferences-4}

\end{document}